\documentclass[11pt]{amsart}
\usepackage[T1]{fontenc}
\usepackage{amsmath,amssymb,amsthm,mathtools}
\usepackage[colorlinks=true,linkcolor=blue,citecolor=blue,urlcolor=blue]{hyperref}
\newtheorem{thm}{Theorem}[section]
\newtheorem{prop}[thm]{Proposition}
\newtheorem{lem}[thm]{Lemma}
\newtheorem{cor}[thm]{Corollary}
\theoremstyle{definition}
\newtheorem{defn}[thm]{Definition}
\newtheorem{exam}[thm]{Example}
\theoremstyle{remark}
\newtheorem{rem}[thm]{Remark}
\newcommand{\E}{\mathcal E}
\newcommand{\F}{\mathcal F}
\newcommand{\T}{\mathcal T}
\newcommand{\Q}{\mathcal Q}
\newcommand{\OO}{\mathcal O}
\newcommand{\PP}{\mathbb P}
\newcommand{\FF}{\mathbb F}
\DeclareMathOperator{\Pic}{Pic}
\DeclareMathOperator{\Supp}{Supp}
\DeclareMathOperator{\rk}{rk}

\DeclareMathOperator{\ev}{ev}
\DeclareMathOperator{\Bl}{Bl}
\title[Hierarchical filtrations and birational geometry]{Hierarchical filtrations of vector bundles and birational geometry}
\author{Rahim Rahmati-asghar}
\address{Department of Mathematics, Faculty of Basic Sciences, University of Maragheh, P. O. Box 55181-83111, Maragheh, Iran}
\email{rahmatiasghar.r@maragheh.ac.ir}
\email{rahmatiasghar.r@gmail.com}
\date{\today}
\begin{document}
\begin{abstract}
We study chains of locally free subsheaves of equal rank whose successive quotients are nonzero divisorial torsion sheaves. Fixing the determinant of the initial term defines a normalized hierarchical depth. An ample intersection argument proves finiteness on every smooth projective variety. On an algebraically closed smooth projective curve, the depth equals the degree of the determinant difference precisely when that difference is effective; otherwise no normalized filtration exists. For birational morphisms of smooth projective surfaces, reflexive pushforward contracts exactly those steps whose determinant divisors are entirely exceptional. This yields a comparison of filtration lengths without a commutativity assumption. Explicit examples show that pullback can strictly increase depth and compute a depth of twenty on a two-point blowup. Finally, we distinguish preservation of evaluation codes for genuinely descended bundles from removal of zero coordinates caused by exceptional fixed components. The coding statements use specified compatible fiber bases or the block Hamming metric.
\end{abstract}
\maketitle

\section{Introduction}
Successive divisorial modifications give chains of subsheaves of a vector bundle, all of the same rank. The effective divisor of the determinant of the composite inclusion constrains the length of such a chain. We retain the term \emph{hierarchical filtration} from the line-bundle framework of \cite{Ra}, with an explicit normalization by the initial determinant.

The intermediate terms are required to be locally free. Their inclusions have torsion quotients and hence are not saturated. This convention allows nontrivial chains while keeping the rank fixed.

An ample-intersection bound proves finiteness on every smooth projective variety. On curves, we obtain an exact formula for arbitrary vector bundles, with effectivity in the Picard group rather than degree alone as the existence criterion. We then prove a surface comparison theorem by extending bundle maps across the finitely many points where a birational morphism is not an isomorphism. Entirely exceptional steps disappear after reflexive pushforward, while every step with a nonexceptional component survives. No reordering of transforms is needed.

Explicit chains show how the determinant bounds are attained. Several modifications may be supported on the same divisor. Blowups may increase depth even for a pullback line bundle with pullback normalization. In a two-point blowup example, an explicit chain and an ample bound give depth twenty.

The final section records consequences for evaluation codes. Genuine pullback preserves the code away from exceptional fibers. On such fibers, evaluations repeat rather than vanish in general. Forced zero coordinates arise in a different situation, when an exceptional divisor is a fixed component of all sections. The projection formula identifies the section spaces, and puncturing the resulting zero coordinates increases relative minimum distance.

\section{Definitions and determinant bounds}
\label{sec:definitions}
Unless specified otherwise, $X$ is a smooth integral projective variety of positive dimension over a field $k$. All sheaves are coherent. A vector bundle is a finite locally free sheaf of positive rank. A subsheaf is saturated when its quotient is torsion-free. Thus a saturated equal-rank inclusion of torsion-free sheaves is an equality. Background on coherent sheaves, reflexive sheaves, and elementary modifications may be found in \cite{Ha,Ha80,HuLe}.

For a torsion-free sheaf $V$ of rank $r$, set $\det V=(\bigwedge^r V)^{\vee\vee}$. For an injective map $u:V\to W$ of vector bundles of equal rank, $\det u$ is a nonzero section of $\det W\otimes(\det V)^{-1}$. Its zero divisor $D(u)$ is effective Cartier and
\[
\det W\simeq\det V\otimes\OO_X(D(u)).
\]
The cokernel is annihilated locally by an equation of $D(u)$, by the adjugate-matrix identity. If it is nonzero, it has pure codimension one: locally it has a free resolution of length one and is Cohen--Macaulay of dimension $\dim X-1$. In particular, a nonisomorphism between equal-rank vector bundles cannot have cokernel supported only in codimension at least two. Determinants of such quotients can equivalently be interpreted using \cite{KnMu}.

\begin{defn}\label{def:hier}
Fix a vector bundle $\E$ of rank $r$ and $\Lambda\in\Pic(X)$. A \emph{hierarchical filtration normalized by $\Lambda$} is a chain
\[
\F_0\subsetneq\F_1\subsetneq\cdots\subsetneq\F_h=\E
\]
of locally free subsheaves of rank $r$, with $\det\F_0\simeq\Lambda$. For each step let $D_i$ be the zero divisor of the determinant of its inclusion. Thus $D_i$ is nonzero effective Cartier and
\[
0\longrightarrow\F_{i-1}\longrightarrow\F_i\longrightarrow\T_i\longrightarrow0,
\qquad
\det\F_i\simeq\det\F_{i-1}\otimes\OO_X(D_i).
\]
The length is $h$. A chain of length zero is allowed exactly when $\det\E\simeq\Lambda$.
\end{defn}

\begin{rem}[Local freeness is not saturation, even in rank one]
For a line bundle $B$ and a nonzero effective Cartier divisor $D$, the subsheaf $B(-D)\subset B$ is locally free of rank one, but its nonzero quotient is torsion. It is therefore not saturated. In fact, the saturation of any equal-rank subsheaf of a torsion-free sheaf is the entire ambient sheaf. The rank-one specialization of Definition~\ref{def:hier} consequently uses inclusions of line bundles with divisorial torsion quotient, not saturated inclusions.
\end{rem}

\begin{defn}\label{def:depth}
The \emph{hierarchical depth} $h_\Lambda(\E)$ is the largest length of a normalized hierarchical filtration, if one exists. If none exists, set $h_\Lambda(\E)=-\infty$. The next proposition proves that the set of possible lengths is bounded, so this definition gives a maximum whenever it is nonempty. We use \emph{maximum-length} rather than inclusion-maximal when referring to a filtration attaining this number.
\end{defn}

\begin{prop}[General finiteness]\label{prop:finite}
Let $H$ be an ample Cartier divisor and put $L=\det\E\otimes\Lambda^{-1}$. Every normalized hierarchical filtration satisfies
\[
L\simeq\OO_X(D_1+\cdots+D_h),\qquad
h\le c_1(L)\cdot H^{n-1},\quad n=\dim X.
\]
In particular, depth is finite when a filtration exists. Existence requires that $L$ have a nonzero section. If $\Pic(X)=\mathbb Z[A]$ for an ample effective divisor $A$ and $L\simeq\OO_X(mA)$, then $h\le m$.
\end{prop}
\begin{proof}
Multiplying the determinant relations gives the first identity. Each nonzero effective divisor has positive integral intersection with $H^{n-1}$; hence
\[
c_1(L)\cdot H^{n-1}=\sum_i D_i\cdot H^{n-1}\ge h.
\]
For the final assertion write $D_i\sim m_iA$. Ampleness gives $m_i\ge1$, and $m=\sum_i m_i\ge h$. For a length-zero chain $L$ is trivial and the assertions also hold.
\end{proof}

\begin{rem}
The definition uses no polarization. The ample divisor in Proposition~\ref{prop:finite} is an auxiliary device for bounding its value. The bound does not imply that depth depends only on the determinant, nor does effectivity alone prove existence in higher rank and dimension. The hypothesis $\Pic(X)=\mathbb Z[A]$ is stronger than numerical Picard rank one.
\end{rem}

\begin{prop}[Twists, inclusions, and direct sums]\label{prop:basic}
The following statements hold whenever the filtrations on the right exist.
\begin{enumerate}
\item For $M\in\Pic(X)$ and $r=\rk\E$,
\[
h_{\Lambda\otimes M^{\otimes r}}(\E\otimes M)=h_\Lambda(\E).
\]
\item If $V\subsetneq W$ are equal-rank vector bundles, then
\[
h_\Lambda(W)\ge h_\Lambda(V)+1.
\]
\item For vector bundles $V,W$ and normalization line bundles $\Lambda_V,\Lambda_W$,
\[
h_{\Lambda_V\otimes\Lambda_W}(V\oplus W)
\ge h_{\Lambda_V}(V)+h_{\Lambda_W}(W).
\]
\end{enumerate}
\end{prop}
\begin{proof}
Tensoring a chain by $M$ and then by $M^{-1}$ gives mutually inverse correspondences, proving (1). Append $W$ to a chain ending at $V$ for (2). For (3), take chains $V_0\subset\cdots\subset V_a=V$ and $W_0\subset\cdots\subset W_b=W$. First use $V_i\oplus W_0$, and then $V\oplus W_j$. The initial determinant is $\Lambda_V\otimes\Lambda_W$, and all quotients are the original nonzero torsion quotients. Neither disjointness nor commutativity is required.
\end{proof}

\begin{rem}
For a nonsplit extension $0\to V\to U\to W\to0$, inverse images of locally free subsheaves $W_j\subset W$ are locally free and give a chain with $V$ fixed. To modify $V$ as well requires compatible extension data.
\end{rem}

\section{Exact computations on curves and split bundles}
\begin{thm}[Curve formula]\label{thm:curve}
Let $C$ be a smooth integral projective curve over an algebraically closed field, $\E$ any vector bundle, and $\Lambda\in\Pic(C)$. Set $L=\det\E\otimes\Lambda^{-1}$. Then
\[
h_\Lambda(\E)=
\begin{cases}
\deg L,&H^0(C,L)\ne0,\\
-\infty,&H^0(C,L)=0.
\end{cases}
\]
\end{thm}
\begin{proof}
A normalized chain makes $L$ effective by Proposition~\ref{prop:finite}, and its length is at most $\deg L$. Conversely, write $L\simeq\OO_C(p_1+\cdots+p_m)$, where repetitions are allowed and $m=\deg L$. Starting with $\E$, successively take the kernel of a surjection of the current bundle to $k(p_j)$. Such a surjection exists by choosing a nonzero functional on its fiber. Each kernel is torsion-free on a smooth curve, hence locally free. Each step lowers the determinant by $p_j$. The last kernel has determinant $\det\E\otimes L^{-1}\simeq\Lambda$. Reversing the chain gives length $m$. If $m=0$, a nonzero section of the degree-zero line bundle $L$ trivializes it, and the trivial chain applies.
\end{proof}

\begin{exam}
On $\PP^1$, $h_{\OO(b)}(\E)=\deg\E-b$ if $\deg\E\ge b$, and is $-\infty$ otherwise. On an elliptic curve, if $\Lambda$ is a nontrivial degree-zero line bundle, then $h_\Lambda(\OO_C)=-\infty$. Equality of degrees cannot replace isomorphism of line bundles.
\end{exam}
\begin{rem}
Over a nonalgebraically closed field, the upper bound by $\deg L$ remains true. A modification at a closed point of degree $e$ consumes determinant degree $e$ in one step. Thus the algebraically closed formula cannot simply be asserted over finite fields or transferred back from a field extension.
\end{rem}

\begin{prop}[Rank one]\label{prop:rankone}
For line bundles $B,\Lambda$ on $X$, $h_\Lambda(B)$ is the largest number of nonzero effective Cartier divisors in a decomposition
\[
B\otimes\Lambda^{-1}\simeq\OO_X(D_1+\cdots+D_h),
\]
with value $-\infty$ if no such decomposition exists, and value zero for the trivial difference when applicable.
\end{prop}
\begin{proof}
The determinant identities give necessity. Conversely, set $B_i=\Lambda\otimes\OO_X(D_1+\cdots+D_i)$ and use the canonical inclusions associated to the effective divisors. Identify the terminal term with $B$.
\end{proof}

\begin{exam}[Intersecting supports do not shorten a chain]\label{ex:projective}
On $\PP^n$, let $\E=\bigoplus_{j=1}^r\OO(a_j)$ and $\Lambda=\OO(b)$, and suppose $m=\sum_j a_j-b\ge0$. Begin with
\[
\F_0=\OO(a_1-m)\oplus\bigoplus_{j=2}^r\OO(a_j).
\]
Choose a hyperplane and increase the first summand by one at each of $m$ steps, using multiplication by its equation. This gives a chain ending at $\E$ and beginning with determinant $\OO(b)$. Proposition~\ref{prop:finite} gives $h_\Lambda(\E)=m$. If $m<0$, no chain exists. The same hyperplane supports every step, so neither disjoint nor distinct supports are necessary.

On $\PP^1\times\PP^1$, if $F$ is a fiber class and $\E=\bigoplus_j\OO(a_jF)$ with $a_j\ge0$, then with normalization $\Lambda=\OO$ modifications of the summands give depth at least $\sum_j a_j$. Intersecting the determinant with an ample divisor of type $(1,1)$ gives the matching upper bound.
\end{exam}

\section{Elementary transforms and reordering}
An elementary transform along an effective Cartier divisor $i:D\hookrightarrow X$ is the kernel of a surjection $V\to i_*G$, where $V$ is locally free and $G$ is locally free on $D$. The kernel is locally free: locally $G$ has projective dimension one over the regular local ring of $X$, and its kernel in a free module is projective. If $G$ has constant rank $s$ on $D$, the determinant decrement is $sD$, not necessarily $D$. More general hierarchical steps allow any quotient arising from an equal-rank bundle inclusion.

\begin{lem}[Intersections and disjoint supports]\label{lem:commute}
Let $\varphi_i:V\to\Q_i$ be surjections, and put $K_i=\ker\varphi_i$. Then, without support assumptions,
\[
\ker(\varphi_2|_{K_1})=K_1\cap K_2=\ker(\varphi_1|_{K_2}).
\]
If $\Supp\Q_1\cap\Supp\Q_2=\varnothing$, both restricted maps are surjective. If $K_1,K_2$ are locally free, their intersection is locally free in this disjoint-support situation.
\end{lem}
\begin{proof}
All three kernels consist of sections annihilated by both maps. At a point in $\Supp\Q_2$, the first quotient vanishes locally, so $K_1=V$ locally and the restricted second map is surjective; elsewhere its target is zero. The other case is symmetric. Locally the intersection equals $K_1$, $K_2$, or $V$, proving local freeness.
\end{proof}
\begin{rem}
Without disjointness, restricted maps need not be surjective onto the original target. In $R=k[x,y]_{(x,y)}$, the map $(x)\to R/(y)$ induced by reduction has image $(x)$ and is not surjective. Kernel equality, retention of specified quotients, and local freeness are separate questions.
\end{rem}
\begin{cor}
In a chain of locally free modifications, adjacent steps with disjoint quotient supports can be interchanged, retaining both quotients, endpoints, normalization, and length. Hence two collections of steps with mutually disjoint supports across collections can be grouped in either order.
\end{cor}
\begin{proof}
Let $\Q_1,\Q_2$ be adjacent quotients. Near the support of $\Q_2$, the first modification is an isomorphism. The second quotient map therefore extends to the original source, by zero away from its support. Apply Lemma~\ref{lem:commute}. The new intermediate kernel is locally free, since locally it equals either the original source or the original terminal kernel. Repeat adjacent exchanges.
\end{proof}

\section{Birational comparison on smooth surfaces}
\label{sec:birational}
In this section $k$ is algebraically closed and $f:Y\to X$ is a birational morphism of smooth projective surfaces. It is an isomorphism over $U=X\setminus Z$ for a finite set $Z$. For a bundle $V$ on $Y$ define
\[
f_*^{\mathrm{ref}}V:=(f_*V)^{\vee\vee}.
\]
This notation denotes reflexive pushforward only, not pullback. It is locally free because $X$ is a smooth surface. The ordinary pushforward $f_*V$ is coherent by properness and torsion-free: multiplication by a nonzero local function on $X$ pulls back to multiplication by a nonzero function on the integral scheme $Y$, and is injective on $V$. The canonical injection $f_*V\hookrightarrow(f_*V)^{\vee\vee}$ is an isomorphism away from finitely many points. This passage is a reflexive extension across codimension two, not saturation in an ambient equal-rank torsion-free sheaf. Taking the reflexive hull can enlarge the space of global sections.

\begin{lem}[Maps and determinant divisors]\label{lem:push}
For an inclusion $V\hookrightarrow W$ of equal-rank bundles on $Y$, with determinant divisor $D$, there is an induced inclusion
\[
f_*^{\mathrm{ref}}V\hookrightarrow f_*^{\mathrm{ref}}W
\]
whose determinant divisor is $f_*D$. It is an isomorphism exactly when $f_*D=0$. Moreover,
\[
\det f_*^{\mathrm{ref}}V\simeq(f_*\det V)^{\vee\vee}.
\]
\end{lem}
\begin{proof}
On $U$ the required morphism is the original inclusion transported through the isomorphism. A morphism between locally free sheaves on a smooth surface extends uniquely across a finite set: its local matrix entries extend by normality. The extension is generically injective; its kernel, being a rank-zero subsheaf of a torsion-free sheaf, is zero. The determinant divisor agrees with $f_*D$ at every codimension-one point, which determines the divisor globally. If this divisor is zero, the determinant is everywhere invertible, so the map is an isomorphism. The final identity follows by the same extension argument for line bundles on $X$, since both sides agree on $U$.
\end{proof}

\begin{thm}[Contraction of a filtration]\label{thm:push}
Let $\F_0\subsetneq\cdots\subsetneq\F_h=\E_Y$ be normalized by $\Lambda_Y$ on $Y$, with determinant divisors $D_i$. Put
\[
\E_X=(f_*\E_Y)^{\vee\vee},\qquad
\Lambda_X=(f_*\Lambda_Y)^{\vee\vee},\qquad
 e_f(\F_\bullet)=\#\{i:f_*D_i=0\}.
\]
Reflexive pushforward followed by deletion of repeated terms gives a hierarchical filtration of $\E_X$, normalized by $\Lambda_X$, of length $h-e_f(\F_\bullet)$. In particular,
\[
\ h\le h_{\Lambda_X}(\E_X)+e_f(\F_\bullet).
\]
No commutativity or disjointness hypothesis is required.
\end{thm}
\begin{proof}
Apply Lemma~\ref{lem:push} to every inclusion. Each resulting term is locally free of the same rank. Precisely the indices with $f_*D_i=0$ give isomorphisms, and all remaining inclusions have nonzero effective determinant divisor. Deleting the isomorphic repetitions therefore gives the claimed length. The determinant of the first term is $\Lambda_X$ by the last identity in the lemma, and the terminal term is $\E_X$. The inequality follows from the definition of depth.
\end{proof}

\begin{cor}[Comparison of depths]\label{cor:depth}
If a normalized filtration of $\E_Y$ exists, let
\[
\begin{aligned}
e_{\max}&=e_{\max}(f;\E_Y,\Lambda_Y)\\
&:=\max\{e_f(\F_\bullet):\F_\bullet\text{ has maximum length for }(\E_Y,\Lambda_Y)\}.
\end{aligned}
\]
Then $h_{\Lambda_Y}(\E_Y)\le h_{\Lambda_X}(\E_X)+e_{\max}$. The maximum exists and is at most $h_{\Lambda_Y}(\E_Y)$.
If $\E_Y=f^*\E$ and $\Lambda_Y=f^*\Lambda$ and a normalized filtration of $\E$ exists, then
\[
h_\Lambda(\E)\le h_{f^*\Lambda}(f^*\E).
\]
\end{cor}
\begin{proof}
The first assertion follows by applying Theorem~\ref{thm:push} to maximum-length chains and using Proposition~\ref{prop:finite}. For the second, pull back every bundle and map of a chain. The pulled-back matrix has nonzero determinant, since $f$ is dominant. It is therefore injective on the integral surface $Y$, and its determinant divisor is $f^*D_i$, which is nonzero effective. This proves the claimed filtration and inequality.
\end{proof}

\begin{rem}
The error term $e_{\max}(f;\E_Y,\Lambda_Y)$ in Corollary~\ref{cor:depth} depends on the contraction $f$, the bundle $\E_Y$, and the normalization $\Lambda_Y$. Changing the normalization changes the family of admissible maximum-length filtrations; changing $f$ changes which divisors are entirely exceptional. The maximum is taken over filtrations for these fixed data. Theorem~\ref{thm:push} identifies the steps removed by contraction. A mixed divisor with both exceptional and nonexceptional components survives; it is not counted by $e_f$.
For a blowup at $p$ with exceptional curve $E$, a strict transform satisfies $\widetilde C=f^*C-\operatorname{mult}_p(C)E$. Absence of an exceptional component does not imply that a divisor is a pullback.
\end{rem}

\begin{cor}[Sequences of surface contractions]
Theorem~\ref{thm:push} applies to the composite of any finite sequence of contractions of $(-1)$-curves between smooth projective surfaces, with the normalization transported by reflexive pushforward.
\end{cor}
\begin{proof}
The composite is a birational morphism of smooth surfaces, so the theorem applies directly.
\end{proof}
\begin{rem}
A sequence of $(-1)$-curve contractions can lead to a minimal smooth surface. This does not always mean that its canonical class is nef: rational and ruled cases must be distinguished from the nonnegative Kodaira-dimension case, and the MMP may end with a Mori fiber space. The comparison theorem applies to any specified sequence of smooth surface contractions; see \cite{Fu} for the MMP framework. Over finite fields, the contractions must be defined over the ground field.
\end{rem}

\subsection{Explicit blowup computations}
\begin{exam}[Strict increase for a pullback]\label{ex:oneblow}
Let $f:Y=\Bl_p\PP^2\to\PP^2$, write $H=f^*\OO_{\PP^2}(1)$ in divisor notation, and let $E$ be exceptional. Downstairs $h_{\OO}(\OO(1))=1$. Upstairs choose a line through $p$ with strict transform $L'\sim H-E$. The decomposition $H=L'+E$ gives
\[
\OO_Y\subsetneq\OO_Y(L')\subsetneq\OO_Y(H).
\]
The divisor $A=2H-E$ is ample on this blowup and $H\cdot A=2$, so Proposition~\ref{prop:finite} gives
\[
h_{\OO_Y}(\OO_Y(H))=2.
\]
One step survives downstairs, while the entirely exceptional step disappears.
\end{exam}

\begin{exam}[The two-point example: depth twenty]\label{ex:twoblow}
Let $f:Y\to\PP^2$ blow up distinct points $p_1,p_2$, with exceptional curves $E_1,E_2$, and put $H=f^*(\text{line})$. Let
\[
\E_Y=\OO_Y(3H+E_1+2E_2)\oplus\OO_Y(2H+E_1+2E_2),
\qquad \Lambda_Y=\OO_Y(E_2).
\]
Thus the determinant difference is $B=5H+2E_1+3E_2$. Let $L$ be the strict transform of the line joining $p_1,p_2$, so $L\sim H-E_1-E_2$. Start with $\OO_Y\oplus\OO_Y(E_2)$. The differences between the terminal and initial summands decompose as follows. In particular, for the second summand the difference is $(2H+E_1+2E_2)-E_2=2H+E_1+E_2$:
\[
3H+E_1+2E_2=3L+4E_1+5E_2,
\qquad
2H+E_1+E_2=2L+3E_1+3E_2.
\]
Increasing one summand at a time by these effective prime divisors gives a normalized chain of length $12+8=20$.

To obtain the upper bound, take $A=3H-E_1-E_2$. This divisor is ample. Its square is $7$, and $A\cdot E_i=1$. For any other irreducible curve, the plane image has degree $d>0$ and multiplicities $m_i\le d$, hence its intersection with $A$ is $3d-m_1-m_2\ge d>0$. The Nakai--Moishezon criterion applies. Since $B\cdot A=20$, Proposition~\ref{prop:finite} gives the matching upper bound:
\[
h_{\OO_Y(E_2)}(\E_Y)=20.
\]
Reflexive pushforward is $\OO_{\PP^2}(3)\oplus\OO_{\PP^2}(2)$, the transported normalization is trivial, and its depth is $5$. In the constructed chain, five steps are supported on $L$ and fifteen are entirely exceptional. Theorem~\ref{thm:push} is sharp for this chain. Counting only the coefficients $2E_1+3E_2$ in the pullback-basis expression would miss the additional exceptional contributions exposed by replacing $H$ with $L+E_1+E_2$.
\end{exam}

\section{Evaluation codes and exceptional divisors}
\label{sec:codes}
Here varieties and morphisms are defined over $\FF_q$; all specified evaluation points are rational. For a rank-$r$ bundle $V$ and ordered set $D=(P_1,\ldots,P_N)$, evaluation has image
\[
C_X(V,D)\subseteq\bigoplus_{i=1}^N V_{P_i}.
\]
Chosen bases of the fibers identify it with a scalar linear code in $\FF_q^{rN}$. Its scalar minimum distance refers to those bases. Alternatively, the block Hamming weight counts nonzero fibers and is invariant under arbitrary changes of fiber bases; its length is $N$. We explicitly preserve bases when asserting equality of scalar distances. In either convention,
\[
\dim C_X(V,D)=h^0(X,V)-\dim\ker(\ev_D),
\]
so equality with $h^0$ requires injective evaluation. For background see \cite{Han,Sa,St}.

\begin{prop}[Codes for genuine pullbacks]\label{prop:codes}
Let $f:Y\to X$ be a birational morphism of smooth projective surfaces over $\FF_q$, and let $V_Y=f^*V$. Then
\[
H^0(Y,V_Y)\simeq H^0(X,V).
\]
For an evaluation set $D$ in the locus where $f$ is an isomorphism, evaluation codes on $D$ and its inverse image are canonically identified as subspaces of the corresponding fiber sums. With transported bases their scalar lengths, dimensions, and minimum distances coincide.
If a contracted exceptional fiber lies over $p\in X(\FF_q)$, evaluations at its rational points all identify with evaluation at $p$. They are repeated blocks, not necessarily zero blocks.
\end{prop}
\begin{proof}
Proper birationality and normality give $f_*\OO_Y=\OO_X$. Projection formula yields $f_*f^*V=V$, proving the section identity. Fiberwise evaluation of a pulled-back section at $y$ is its original value at $f(y)$. This proves both assertions.
\end{proof}

\begin{rem}
Trivial determinant contribution along the exceptional locus does not by itself imply that a bundle descends. For a single blowup, $\OO_Y(E)\oplus\OO_Y(-E)$ has trivial determinant but restricts to $\OO_E(-1)\oplus\OO_E(1)$, whereas a pullback restricts trivially to $E$. Also, deleting repeated blocks can change minimum distance or dimension if all copies of an independent evaluation are deleted.
\end{rem}

\begin{prop}[Exceptional fixed components]\label{prop:zero}
Let $f:Y\to X$ be the blowup of a smooth rational point $p$, with exceptional curve $E$, and let $V$ be a rank-$r$ vector bundle on $X$. For any integer $a\ge1$, put $W=f^*V\otimes\OO_Y(aE)$. Then
\[
f_*W\simeq V,\qquad H^0(Y,W)\simeq H^0(X,V),
\]
and every global section of $W$ vanishes along $E$ as a section of $W$.
For evaluation sets $D'\subset Y(\FF_q)\setminus E$ and $D_E\subset E(\FF_q)$, compatible bases give
\[
C_Y(W,D'\cup D_E)=C_X(V,f(D'))\times\{0\}^{r|D_E|}.
\]
\end{prop}
\begin{proof}
A section of $\OO_Y(aE)$ is a rational function regular away from $E$, with poles allowed only along $E$. It descends to a function regular on $X\setminus\{p\}$, and hence extends across $p$ by normality. Therefore $f_*\OO_Y(aE)=\OO_X$; projection formula proves the first identities. The natural inclusion $f^*V\to W$ induces an isomorphism on global sections, since its pushforward is the identity of $V$. Locally this inclusion is multiplication by an equation of $E$ to the power $a$ in a frame of $W$. Its image therefore vanishes along $E$. Away from $E$ use the canonical trivialization of $\OO_Y(aE)$ to identify evaluations and transport fiber bases. On $E$ all evaluation blocks are zero.
\end{proof}

\begin{cor}[Puncturing zero coordinates]
In Proposition~\ref{prop:zero}, suppose the code on $D'$ is nonzero, $|D'|=N>0$, and $|D_E|=t>0$. Puncturing the zero blocks preserves dimension and minimum distance. For scalar distance $d$, the relative distances before and after puncturing are $d/(r(N+t))$ and $d/(rN)$, respectively. The ratio of relative distance after puncturing to relative distance before puncturing is $(N+t)/N$. Writing $d_{\mathrm{block}}$ for the common block minimum distance, the block relative distances are $d_{\mathrm{block}}/(N+t)$ and $d_{\mathrm{block}}/N$, respectively, with the same after-to-before ratio $(N+t)/N$.
\end{cor}
\begin{proof}
Appending or deleting identically zero coordinates changes no nonzero word's weight and does not change the dimension. The formulas follow.
\end{proof}

\begin{exam}[A verified small code]
Let $X=\PP^2_{\FF_q}$ with $q\ge3$, let $p=[1:0:0]$, and let $f:Y\to X$ be its blowup. Take the line $\ell=\{x_0=0\}$, which avoids $p$, and distinct nonzero $a,b\in\FF_q$. Evaluate $V=\OO_X(1)$ at
\[
[0:1:0],\quad[0:1:a],\quad[0:1:b],\quad[0:0:1].
\]
Using the displayed representatives as fiber trivializations, the restrictions of the sections $x_1,x_2$ give the two rows of the generator matrix below. The section $x_0$ restricts to zero on $\ell=\{x_0=0\}$ and hence contributes no additional row. Thus the code has generator matrix
\[
G=\begin{pmatrix}1&1&1&0\\0&a&b&1\end{pmatrix}.
\]
It has dimension $2$. A nonzero linear form on $\ell\simeq\PP^1$ has at most one zero among the four distinct points, and some such forms have exactly one zero. Its parameters are therefore $[4,2,3]_q$. Here $h^0(X,V)=3$ and the evaluation kernel is the one-dimensional span of $x_0$.

Now take $W=f^*\OO_X(1)\otimes\OO_Y(E)$ and add any $t$ distinct rational points of $E$, where $1\le t\le q+1$. Proposition~\ref{prop:zero} gives a code with parameters $[4+t,2,3]_q$. Deleting the $t$ zero coordinates raises relative distance from $3/(4+t)$ to $3/4$. The improvement is precisely zero-coordinate puncturing.
\end{exam}

\begin{rem}[The sign of the exceptional twist]
For the blowup of a smooth point, $\OO_Y(E)|_E=\OO_E(-1)$ and $(f^*L)|_E$ is trivial for every line bundle $L$ on $X$. Thus
\[
(f^*L\otimes\OO_Y(-E))|_E=\OO_E(1),
\qquad
(f^*L\otimes\OO_Y(E))|_E=\OO_E(-1).
\]
The negative exceptional twist imposes vanishing at $p$ on sections downstairs, but does not imply intrinsic vanishing on $E$ upstairs. For example, sections of $\OO_Y(H-E)$ restrict nontrivially to $E$. The positive twist used in Proposition~\ref{prop:zero} produces the exceptional fixed component. These evaluation properties are determined by descent and by the restriction of sections to the exceptional curve.
\end{rem}

\section{Further questions}
The curve formula reduces depth to the effectivity and degree of the determinant difference. For nonsplit bundles in higher dimensions, one must also construct the intermediate locally free sheaves. It would be useful to characterize the bundles for which the determinant bound in Proposition~\ref{prop:finite} is attained.

The surface comparison raises a second question: can $e_{\max}(f;\E_Y,\Lambda_Y)$ be bounded directly in terms of the bundle and the exceptional curves? Example~\ref{ex:twoblow} shows that the coefficients of the determinant difference in a pullback basis do not suffice for a naive count. A useful bound must also account for effective representatives of the nonexceptional divisor classes.

\section*{Acknowledgments}
The author is grateful to Professor Pierre Deligne for valuable comments that led to a clarification and correction of the definition of hierarchical filtration.

\end{document}